\documentclass[12pt,a4paper]{article}
\usepackage{amsmath}
\usepackage{amssymb}
\usepackage{amsthm}
\usepackage{amsfonts}
\usepackage{mathscinet}
\usepackage{enumerate}
\usepackage{pdfsync}
\usepackage{setspace}
\usepackage{mathrsfs}
\usepackage{tikz}
\usepackage{stmaryrd}
\usepackage{hyperref}
\usetikzlibrary{matrix}

\usepackage[all]{xypic}

\newcommand \lh {\text{lh}}

\newcommand {\new} {\newcommand}
\newcommand {\renew} {\renewcommand}
\newcommand \opera[2] {\renew #1 {\operatorname{#2}}}
\newcommand \oper[2] {\new #1 {\operatorname{#2}}}

\newcommand \gcode[1] {\ulcorner\! #1 \!\urcorner}	
\newcommand \se [1] { \{ #1 \}}			
\newcommand\set [2]{ \{#1:#2\} }			
\newcommand\res {\!\upharpoonright\!}		

\newcommand\eqiv {\leftrightarrow}			
\newcommand\seq [2]{ \langle #1:#2 \rangle }	
\newcommand\corner[1] {
  \langle #1 \rangle}		

\newcommand\coll{{Coll}}

\newcommand\concat {{^\frown}}   

\oper{\Ord}{Ord}				

\oper{\ZFC}{ZFC}				
\oper{\rank}{rank}				
\oper{\crit}{cr}					
\oper{\crt}{crt}					
\oper{\cf}{cf}					
\oper{\height}{ht}				
\oper{\wfcore}{wfcore}					
\oper{\core}{core}					

\oper{\Ult}{Ult}				
\oper{\Cone}{Cone}				
\oper{\dirlim}{dirlim}
\oper{\rud}{rud}		
\oper{\const}{const}
\oper{\OD}{OD}
\oper{\final}{final}
\oper{\HYP}{HYP}
\oper{\wfp}{wfp}
\oper{\Hom}{Hom}
\new{\ult} {\Ult}
\oper{\dom}{dom}
\oper{\rep}{rep}

\oper{\suc}{succ}
\oper{\fac}{fac}

\oper{\Code}{Code}

\oper{\ran}{ran}
\oper{\maxdom}{maxdom}
\oper{\maxran}{maxran}

\oper{\lp}{Lp} 
\oper{\pro}{pro} 
\oper{\lift}{lift} 
\opera{\drop}{drop} 
\oper{\base}{base} 

\newcommand\iniseg {\vartriangleleft}		

\newtheorem{theorem}{Theorem}[section]

\newtheorem{conjecture}[theorem]{Conjecture}
\newtheorem{lemma}[theorem]{Lemma}

\newtheorem{claim}[theorem]{Claim}
\theoremstyle{definition}
\newtheorem{question}[theorem]{Question}
\newtheorem{definition}[theorem]{Definition}

\theoremstyle{definition} \newtheorem*{defplain*}{\gy{Definition}}
\theoremstyle{definition} \newtheorem*{thmplain*}{\gy{Theorem}}
\theoremstyle{definition} \newtheorem*{queplain*}{\gy{Question}}
\theoremstyle{definition} 
\theoremstyle{definition} \newtheorem*{corplain*}{\gy{Corollary}}

\theoremstyle{remark} 		
\theoremstyle{remark}

\newcommand{\game}{{\Game}}

\newcommand{\boldsigma}[1]{{\boldsymbol{\Sigma}^1_{#1}}}

\newcommand{\bolddelta}[1]{{\boldsymbol{{\delta}}^1_{#1}}}
\newcommand{\boldDelta}[1]{{\boldsymbol{{\Delta}}^1_{#1}}}

\newcommand{\WO}{{\textrm{WO}}}

\newcommand{\DEF}{=_{{\textrm{DEF}}}}

\newcommand{\Diff}{\operatorname{Diff}}

\newcommand{\ot}{\mbox{o.t.}}

\newcommand{\sharpcode}[1]{  \left|   #1  \right|}

\newcommand{\comm}[1]{{}}

\usepackage{bbm, dsfont,authblk}
\title{Iterates of $M_1$}
\date{\today{}}
\author{Yizheng Zhu}

\affil{Institut f\"{u}r mathematische Logik und Grundlagenforschung \\
Fachbereich Mathematik und Informatik\\
Universit\"{a}t M\"{u}nster\\
Einsteinstr. 62 \\
48149 M\"{u}nster, Germany 
}

\begin{document}
\maketitle{}

\abstract{Assume $\boldDelta{2}$-determinacy. Let $L_{\kappa_3}[T_2]$ be the admissible closure of the Martin-Solovay tree and let $M_{1,\infty}$ be the direct limit of all iterates of $M_1$ via countable trees. We show that $L_{\kappa_3}[T_2] \cap V_{u_{\omega}} $ is the universe of $M_{1,\infty} | u_{\omega}$. 
}

\section{Introduction}
\label{sec:introduction}

Canonical models naturally arise in models of determinacy. Moschovakis et al.\  \cite[Section 8G]{mos_dst} started the investigation of the models $H_{\Gamma}$ and $L[T_{\Gamma}]$ if $AD$ holds and $\Gamma$ is a scaled pointclass closed under $\forall^{\mathbb{R}}$. 
These models have set-theoretical identity which are useful in further study of regularity properties of sets of reals.
At projective levels, when $\Gamma$ is $\Pi^1_{2n+1}$, the model is $H_{\Gamma^1_{2n+1}} = L[T_{2n+1}]$, shown by Becker-Kechris \cite{becker_kechris_1984}, where $T_{2n+1}$ is the tree of the $\Pi^1_{2n+1}$-scale on a good universal $\Pi^1_{2n+1}$ set. The next obvious question to ask is the internal structure of $H_{\Gamma}$, e.g.\ does GCH hold? 

For projective levels, Steel \cite{steel_hodlr_1995} shows that $L[T_{2n+1}]$ is a mouse. Let $M_{n,\infty}^{\#}$ be the direct limit of all the countable iterates of $M_{n,\infty}$ based on the bottom Woodin of $M_{n,\infty}^{\#}$, let $M_{n,\infty}$ be the result of iterating the top extender of $M_{n,\infty}^{\#}$ out of the universe. Let $\delta_{n,\infty}$ be the bottom Woodin of $M_{n,\infty}$ and $\kappa_{n,\infty}$ be the least $<\delta_{n,\infty}$ strong cardinal in $M_{n,\infty}$. Steel shows that $\kappa_{2n,\infty} = \bolddelta{2n+1}$ and that the universe of $M_{2n,\infty} | \kappa_{2n,\infty} $ is $L_{\bolddelta{2n+1}}[T_{2n+1}]$. It is worth mentioning that the extender sequence of $M_{2n,\infty} | \kappa_{2n,\infty} $ is definable over the universe of  $M_{2n,\infty} | \kappa_{2n,\infty} $:  the universe of  $M_{2n,\infty} | \kappa_{2n,\infty} $ satisfies that ``I am closed under the $M_{2n-1}^{\#}$-operator, there is no inner model with $2n$ Woodin cardinals, and I am the relativized Jensen-Steel core model (\cite{jensen_steel_core_model,core_model_induction})'', and the extender sequence of the Jensen-Steel core model built in the universe $M_{2n,\infty} | \delta_{2n+1,\infty} $ coincides with the extender sequence of $M_{2n,\infty} | \delta_{2n,\infty} $.
This paves the way for the study of the canonical model $L[T_{2n+1}]$ using inner model theory.
It is a strong evidence that $M_2$ is the correct model to work with for further investigation of $\Sigma^1_4$ sets. 
What about $M_1$? What does its direct limit $M_{1,\infty}$ look like? A partial result was by Hjorth \cite{hjorth_boundedness_lemma}, that $\delta_{1,\infty}=u_{\omega}$. This paper shows that the structure of $M_{1,\infty}$ has a  canonical characterization from descriptive set theory. 
The odd levels and even levels can now be unified with the following scope. 

Assume AD. Consider the Suslin cardinals. The first few are $\omega, \omega_1, u_{\omega}, \bolddelta{3}, (\bolddelta{5})^{-}, \bolddelta{5},\cdots$. For every Suslin cardinal $\kappa$, the pointclass of $\kappa$-Suslin sets is closed under $\exists^{\mathbb{R}}$. For the first few, $\omega$-Suslin sets are $\boldsigma{1}$, $\omega_1$-Suslin sets are $\boldsigma{2}$, $u_{\omega}$-Suslin sets are $\boldsigma{3}$, $\bolddelta{3}$-Suslin sets are $\boldsigma{4}$, etc. Consider the associated lightface pointclass in each case and consider a nice coding system of each Suslin cardinal. We can build canonical models associated to each Suslin cardinal in the following way:
\begin{enumerate}
\item Ordinals in $\omega_1$ have a $\Pi^1_1$-coding system, namely $\WO$, the set of wellorderings on $\omega$. $\WO$ is a $\Pi^1_1$ set and $\sharpcode{\cdot}$ is a $\Pi^1_1$-norm of $WO$ onto $\omega_1$. Define the universal $\Sigma^1_2$ set of ordinals in $\omega_1$ relative to this coding:
  \begin{displaymath}
    \mathcal{O}_{\Sigma^1_2} = \set{ ( \gcode{\varphi},\alpha)}{ \varphi \text{ is $\Sigma^1_2$}, \exists x\in \WO~(\sharpcode{x} = \alpha \wedge \varphi(x))}.
  \end{displaymath}
The canonical model associated to $\omega_1$ is $L_{\omega_1}[\mathcal{O}_{\Sigma^1_2}]$. By Shoenfield absoluteness, this model is just $L_{\omega_1}$.
\item Ordinals in $u_\omega$ have a $\Delta^1_3$-coding system, namely $\WO_{\omega}$, the set of sharp codes for ordinals in $u_{\omega}$.   $\sharpcode{\cdot}$ is a $\Delta^1_3$-norm of $\WO_{\omega}$ onto $u_\omega$. Define the universal $\Sigma^1_3$ set of ordinals in $u_\omega$ relative to this coding:
  \begin{displaymath}
    \mathcal{O}_{\Sigma^1_3} = \set{ ( \gcode{\varphi},\alpha)}{ \varphi \text{ is $\Sigma^1_3$}, \exists x\in \WO_{\omega}~(\sharpcode{x} = \alpha \wedge \varphi(x))}.
  \end{displaymath}
The canonical model associated to $u_\omega$ is $L_{u_\omega}[\mathcal{O}_{\Sigma^1_3}]$. By Q-theory and Kechris-Martin \cite{Q_theory,becker_kechris_1984,Kechris_Martin_I,Kechris_Martin_II}, the universe of this model equals to $L_{\kappa_3}[T_2]\cap V_{u_{\omega}}$, where $L_{\kappa_3}[T_2]$ is the admissible closure of the Martin-Solovay tree $T_2$. The main theorem of this paper is
\begin{displaymath}
  L_{u_{\omega}}[\mathcal{O}_{\Sigma^1_3}]  \text{ and } M_{1,\infty} | \delta_{1,\infty} \text{ have the same universe}. 
\end{displaymath}
There is a small difference between $L_{u_{\omega}}[\mathcal{O}_{\Sigma^1_3}] $ and $M_{1,\infty} | \delta_{1,\infty}$. 
Just like the case with $M_{2n,\infty}$,  the extender sequence of $M_{2n+1,\infty} | \delta_{2n+1,\infty} $ is definable over the universe of  $M_{2n+1,\infty} | \delta_{2n+1,\infty} $:  the universe of  $M_{2n+1,\infty} | \kappa_{2n+1,\infty} $ satisfies that ``I am closed under the $M_{2n}^{\#}$-operator, there is no inner model with $2n+1$ Woodin cardinals, and I am the relativized Jensen-Steel core model (\cite{jensen_steel_core_model})'', and the extender sequence of the Jensen-Steel core model built in the universe $M_{2n+1,\infty} | \kappa_{2n+1,\infty} $ coincides with the extender sequence of $M_{2n+1,\infty} |\delta_{2n+1,\infty} $. However, $\mathcal{O}_{\Sigma^1_3}$ is \emph{not} definable over the universe of $ L_{u_{\omega}}[\mathcal{O}_{\Sigma^1_3}]$. This is because the universe of $L_{u_{\omega}}[\mathcal{O}_{\Sigma^1_3}]$ is a model of ZFC, while using the predicate $\mathcal{O}_{\Sigma^1_3}$, one can easily define the sequence $(u_n: n < \omega)$ which singularizes $u_{\omega}$. 
\item  Ordinals in $\bolddelta{3}$ have a $\Pi^1_3$-coding system. Take a good universal $\Pi^1_3$ set $G$ and a $\Pi^1_3$ norm $\psi : G \to \bolddelta{3}$. Define the universal $\Sigma^1_4$ set of ordinals in $\bolddelta{3}$ relative to this coding:
  \begin{displaymath}
    \mathcal{O}_{\Sigma^1_4} = \set{ ( \gcode{\varphi},\alpha)}{ \varphi \text{ is $\Sigma^1_4$}, \exists x\in G~(\psi(x) = \alpha \wedge \varphi(x))}.
  \end{displaymath}
The canonical model associated to $\bolddelta{3}$ is $L_{\bolddelta{3}}[\mathcal{O}_{\Sigma^1_4}]$. It is independent of the choice of $G$ and $\varphi$, shown by Moschovakis \cite[8G.22]{mos_dst}. 
Steel \cite{steel_hodlr_1995} shows that $M_{2,\infty} | \kappa_{2,\infty} $ and $  L_{\bolddelta{3}}[\mathcal{O}_{\Sigma^1_4}]$ have the same universe. Here, in constast to $M_{1,\infty}|\delta_{1,\infty}$, the extender sequence of $M_{2,\infty}|\kappa_{2,\infty}$ and $\mathcal{O}_{\Sigma^1_4}$ are both definable over the universe of $M_{2,\infty}| \kappa_{2, \infty}$. 
\end{enumerate}
We mention without the proof that this paper routinely generalizes to
the higher levels based on
\cite{sharpI,sharpII,sharpIII,sharpIV}. Under AD, for arbitrary $n$, there is a $\Delta^1_{2n+1}$ coding system of ordinals in $(\bolddelta{2n+1})^-$ which generalizes the $\WO_{\omega}$ coding of $u_{\omega}$. Define $\mathcal{O}_{\Sigma^1_{2n+1}}$, the universal $\Sigma^1_{2n+1}$ subset of $(\bolddelta{2n+1})^-$ relative to this coding. Then
\begin{displaymath}
  M_{2n-1,\infty} | (\bolddelta{2n+1})^- \text{ and } L_{(\bolddelta{2n+1})^-}[\mathcal{O}_{\Sigma^1_{2n+1}}] \text{ have the same universe}.
\end{displaymath}

This unification of odd and even levels should hopefully isolate the correct questions.  For instance, the model $L[T_2]$, and its generalizations, $L[T_{2n}]$, 
were considered ``canonical'' \cite{MR730596,cabal_4_survey}. The uniqueness of $L[T_{2n}]$ was asked in \cite{MR730596} and solved by Hjorth \cite{hjorth_LT2} for $n=1$ and Atmai \cite{atmai_thesis} for arbitrary $n$. 
Atmai-Sargsyan \cite{atmai_thesis} proves that $L[T_2] = L[M_{1,\infty}^{\#}]$. However, it is hard and unnatural to investigate this model, the fundamental reason being that this model is the result of constructing on top of a non-sound mouse $M_{1,\infty}^{\#}$. Most of the standard methods in inner model theory break down as we always construct on top of a sound mouse.  It might seem as if inner model theory is not good enough to study $L[T_2]$. However, this is not the right intuition. It is inner model theory that helps figuring out the correct model. Atmai-Sargsyan's result suggests that $L[T_2]$ is the wrong model to work with, and this paper finds the correct model: $L_{u_{\omega}}[\mathcal{O}_{\Sigma^1_3}]$. This local version of $L[T_2]$ is a mouse. It deserves more attention. For instance, it captures $\Sigma^1_3$-truth by Q-theory \cite{Q_theory}:
\begin{enumerate}
\item There is an effective map $\varphi \mapsto \varphi^{*}$ that sends a $\Sigma^1_3$ formula $\varphi$ to a $\Pi^1_3$ formula $\varphi^{*} $such that $V \models \varphi$ iff $M_{1,\infty}\models \varphi^{*}$.
\item There is an effective map $\varphi \mapsto \varphi_{*}$ that sends a $\Sigma^1_3$ formula $\varphi$ to a $\Pi^1_3$ formula $\varphi_{*}$ such that  $M_{1,\infty}\models \varphi$ iff $V \models \varphi_{*}$.
\end{enumerate}
This anti-$\Sigma^1_3$-correctness result is comparable to the $\Sigma^1_{2n}$-correctness of the model $L[T_{2n+1}]$.

Under AD, there should be a canonical model associated to every Suslin cardinal. The next Suslin cardinal beyond projective is $\bolddelta{\omega}= \sup_{n<\omega}\bolddelta{n}$. $\bolddelta{\omega}$-Suslin sets are $\boldsymbol{\Sigma}_2^{J_2(\mathbb{R})}$. The canonical model should be $L_{\bolddelta{\omega}}[\mathcal{O}_{\Sigma_2^{J_2(\mathbb{R})}}]$. This model should also have a similar fine structure as in the projective levels.
However, it is still an open question whether the set of reals in this model is a mouse set, cf. \cite[Section 8.4]{Sandra}.

\section{Q-theory}
\label{sec:q-theory}
We assume $\boldDelta{2}$-determinacy throughout this paper. 
This section is a brief overview of the $Q$-theory in \cite{Q_theory} and related papers. $\WO = \WO_1$ is the set of canonical codes for countable ordinals. $\WO$ is $\Pi^1_1$. For $0<n<\omega$,
\begin{displaymath}
  \WO_{n+1} = \set{ \corner{\gcode{\tau}, y^{\#}}}{ \tau \text{ is a $n+1$-ary Skolem term}}
\end{displaymath}
$\WO_{n+1}$ is $\Pi^1_2$. 
For $\corner{\gcode{\tau}, y^{\#}} \in \WO_{n+1}$, it codes an ordinal below $u_{n+1}$:
\begin{displaymath}
  \sharpcode{\corner{\gcode{\tau}, y^{\#}}  } = \tau^{L[y]}(y,u_1,\dots,u_{n}).
\end{displaymath}
$\WO_{\omega} = \cup_{1 \leq n < \omega}\WO_n$. $A \subseteq u_{\omega} \times  \mathbb{R} $ is said to be $\Sigma^1_3$ iff
\begin{displaymath}
  \set{(w,x)}{ w \in \WO_{\omega}, (x,\sharpcode{w}) \in A}
\end{displaymath}
is $\Sigma^1_3$. Similarly define $\Pi^1_3$, $\Delta^1_3$ and their relativizations. $T_2$ is the Martin-Solovay tree on $\omega \times u_{\omega}$ projecting to $\set{x^{\#}}{x \in \mathbb{R}}$. $T_2$ is a $\Delta^1_3$ subset of $(\omega \times u_{\omega})^{<\omega}$. $\kappa_3^x$ is the least admissible ordinal over $(T_2,x)$. $\kappa_3 = \kappa_3^0$. A model-theoretic representation of $\Pi^1_3$ subsets is:
\begin{theorem}[\cite{becker_kechris_1984,Kechris_Martin_I,Kechris_Martin_II}]
  \label{thm:becker_kechris_martin}
Assume $\boldDelta{2}$-determinacy. 
  Suppose $A \subseteq  u_{\omega} \times \mathbb{R}$. The following are equivalent.
  \begin{enumerate}
  \item $A$ is $\Pi^1_3$.
  \item There is a $\Sigma_1$-formula $\varphi$ such that $(\alpha,x) \in A$ iff $L_{\kappa_3^x}[T_2,x] \models \varphi(T_2, \alpha, x)$. 
  \end{enumerate}
\end{theorem}
The conversions between the $\Pi^1_3$ definition of $A$ and the $\Sigma_1$-formula $\varphi$ in Theorem~\ref{thm:becker_kechris_martin} are effective.

\section{Suitable Premice}
\label{sec:suitable-premice}

This section contains a brief overview of the usual definitions on suitable premice that occurs in a typical HOD computation (cf. \cite{hod_as_a_core_model}). 

If $N$ has a unique Woodin cardinal, it is denoted by $\delta^N$. The extender algebra in $N$ at $\delta^N$  with $\omega$-generators is denoted by $\mathbb{B}^N$. 
 A class-sized premouse $N$ is \emph{$M_1$-like} iff there is $\delta$ such that $N = L[N|\delta]$ and
\begin{enumerate}
\item $N \models \delta$ is Woodin,
\item for every $\eta < \delta$, $L[N| \eta] \models $``$\eta$ is not Woodin'', and
\item $N \models \forall \eta < \delta (\text{``I am $(\eta,\eta)$-iterable}")$.  
\end{enumerate}
A premouse $\mathcal{P}$ is \emph{suitable} iff $L[\mathcal{P}]$ is $M_1$-like and $o(\mathcal{P}) $ is the cardinal successor of $\delta^{\mathcal{P}}$ in $L[\mathcal{P}]$. 
If $\mathcal{P}$ is suitable, $\mathcal{P}$ is called the suitable initial segment of $L[\mathcal{P}]$.
The suitable initial segment of an $M_1$-like $N$ is called $N^{-}$. 
 The set of reals coding countable, suitable premice is $\Delta^1_3$.

 If $\mathcal{T}$ is a normal iteration tree on a suitable $\mathcal{P}$, then
\begin{enumerate}
\item $\mathcal{T}$ is \emph{short} iff either $\mathcal{T}$ has a last model $\mathcal{M}_{\alpha}$ such that $\mathcal{M}_{\alpha}$ is suitable or $[0,\alpha]_T$ drops, or $\mathcal{T}$ has limit length , $\mathcal{Q}(\mathcal{T})$ exists, and $\mathcal{Q}(\mathcal{T}) \iniseg L[\mathcal{M}(\mathcal{T})]$.
\item $\mathcal{T}$ is \emph{maximal} iff $\mathcal{T}$ is not short. 
\end{enumerate}
If $\mathcal{P}$ is suitable, then $\mathcal{P}$ is \emph{short tree iterable} iff whenever $\mathcal{T}$ is a short tree on $\mathcal{P}$, then
\begin{enumerate}
\item if $\mathcal{T}$ has a last model, then it can be freely extended by one more ultrapower, that is, every putative normal tree $\mathcal{U}$ extending $\mathcal{T}$ and having length $\lh(\mathcal{T})+1$ has a wellfounded last model, and moreover this model is suitable if the branch leading to it does not drop,
\item if $\mathcal{T}$ has limit length and $\mathcal{T}$ is short, then $\mathcal{T}$ has a cofinal wellfounded branch $b$, and moreover $\mathcal{M}^{\mathcal{T}}_b$ is suitable if $b$ does not drop. 
\end{enumerate}
It is shown in \cite{hod_as_a_core_model} that every suitable $\mathcal{P}$ is short tree iterable. If $\mathcal{P}$ is suitable, $\mathcal{Q}$ is called a \emph{pseudo-normal-iterate} of $\mathcal{P}$ iff $\mathcal{Q}$ is suitable, and there is a normal tree $\mathcal{T}$ on $\mathcal{P}$ such that either $\mathcal{Q}$ is the last model of $\mathcal{T}$, or $\mathcal{T}$ is maximal and $\mathcal{Q}$ is the suitable initial segment of $L[\mathcal{M}(\mathcal{T})]$. 

Suppose $s$ is a finite set of ordinals. We define $s^{-} = s \setminus \max(s)$ and  $\gamma^{\mathcal{P}}_s = \sup (Hull^{J_s[\mathcal{P}]}(s^{-}) \cap \delta^{\mathcal{P}})$. If $\mathcal{T}$ is an iteration tree on $J_{\max(s)}[\mathcal{P}]$ with two cofinal branches $b,c$ such that $\mathcal{M}_b^{\mathcal{T}} = \mathcal{M}_c^{\mathcal{T}} = J_{\max(s)} [\mathcal{M}(\mathcal{T})]$ and $\pi_b^{\mathcal{T}}(s^{-}) = \pi_c^{\mathcal{T}}(s^{-})$, then
\begin{displaymath}
\pi_b^{\mathcal{T}} \res \gamma^{\mathcal{P}}_s = \pi_c^{\mathcal{T}} \res \gamma^{\mathcal{P}}_s.
\end{displaymath}
This is a useful consequence of the zipper argument in \cite[Theorem 6.10]{steel-handbook}. It is used by Hjorth \cite{hjorth_boundedness_lemma} to show that $u_{\omega} = \delta^{M_{1,\infty}}$.

\section{The full direct limit $M_{1,\infty}$}
\label{sec:full-direct-limit}

\begin{definition}\label{def:skolem_term}
We define a fixed binary Skolem term 
\begin{displaymath}
  \rho
\end{displaymath}
as follows. 
  If $\mathcal{P}$ is a countable, suitable premouse, $n<\omega$, for countable ordinals $\alpha_1 < \dots < \alpha_{n}$, define the bad-sequence relation
  \begin{displaymath}
    ( \seq{{\mathcal{T}}_i  }{i < k'}, \seq{\mathcal{P}_i}{i \leq k'}, \eta ') <_{\alpha_1,\dots,\alpha_{n}}^{\mathcal{P}}  ( \seq{{\mathcal{U}}_i  }{i < k}, \seq{\mathcal{Q}_i}{i \leq k}, \eta )
  \end{displaymath}
  iff
  \begin{enumerate}
  \item $k \leq k' < \omega$,
  \item $\forall i < k ({\mathcal{T}}_i = {\mathcal{U}}_i )$, $\forall i \leq k (\mathcal{P}_i = \mathcal{Q}_i)$,
  \item $\mathcal{P}_0 = J_{\alpha_{n}}[\mathcal{P}]$,
  \item for any $i < k'$,  ${\mathcal{T}}_i$ is a countable, normal iteration tree on $J_{\alpha_n}[\mathcal{P}_i]$ with last model $J_{\alpha_n}[\mathcal{P}_{i+1}]$ such that  $\pi^{{\mathcal{T}}_i}$ exists and  $\pi^{{\mathcal{T}}_i}(\alpha_1,\dots,\alpha_{n-1}) = (\alpha_1,\dots,\alpha_{n-1})$,
  \item $\eta < \gamma^{\mathcal{P}_k}_{\se{\alpha_1,\dots,\alpha_{n}}}$, $\eta' < \gamma^{\mathcal{P}_{k'}}_{\se{\alpha_1,\dots,\alpha_{n}}}$, $\eta' < \pi^{\oplus_{k \leq i < k'} {\mathcal{T}}_i } (\eta)$.
  \end{enumerate}
  $<_{\alpha_1,\dots,\alpha_{n}}^{\mathcal{P}}$ is $\Delta^1_1$ in the
  codes of $\mathcal{P}$ and $\alpha_1,\dots,\alpha_{n}$. The bad
  sequence argument in \cite{hjorth_boundedness_lemma} shows that
  $<_{\alpha_1,\dots,\alpha_{n}}^{\mathcal{P}}$ is wellfounded for any
  countable $\alpha_1< \dots < \alpha_{n} $. Hence, the rank of
  $<_{\alpha_1,\dots,\alpha_{n}}$ is smaller than the smallest
  $(\mathcal{P},\alpha_{n})$-admissible. By Shoenfield absoluteness, 
for any $\nu < \gamma^{\mathcal{P}}_{\se{\alpha_1, \dots, \alpha_n}}$,
the rank of $(\emptyset, \corner{J_{\alpha_{n}}[\mathcal{P}]}, \nu)$
in $<_{\alpha_1,\dots,\alpha_{n}}^{\mathcal{P}}$ is the same in any
proper class model $W$ of ZFC satisfying that $(\mathcal{P}, \alpha_1, \dots,
\alpha_n) \in HC^W$. 
There is a fixed Skolem term $\rho$ such that for $\nu < \gamma^{\mathcal{P}}_{\se{\alpha_1, \dots, \alpha_{n}}}$,
  \begin{align*}
    \rho^{L[\mathcal{P}]}(\nu, (\alpha_1,\dots,\alpha_{n})) =  ~& \text{the rank of $(\emptyset, \corner{J_{\alpha_{n}}[\mathcal{P}]}, \nu)$ in $<_{\alpha_1,\dots,\alpha_{n}}^{\mathcal{P}}$}\\
    & \text{as computed in }L[\mathcal{P}]^{\coll(\omega, \alpha_n)}.
  \end{align*}
\end{definition}
Thus, for any proper class model $W$ of ZFC satisfying that $(\mathcal{P}, \alpha_1, \dots,
\alpha_n) \in HC^W$, $ \rho^{L[\mathcal{P}]}(\nu, (\alpha_1,\dots,\alpha_{n})) $ is the rank of $(\emptyset, \corner{J_{\alpha_{n}}[\mathcal{P}]}, \nu)$ in $<_{\alpha_1,\dots,\alpha_{n}}^{\mathcal{P}}$ as computed in $W$. 
This fixed term $\rho$ is thus allowed to apply on uncountable
ordinals $\alpha_1,\dots,\alpha_{n}$ as well. For instance, when $\mathcal{P}$ is still countable in
$V$, 
\begin{displaymath}
  \rho^{L[\mathcal{P}]} (\nu, (u_1, \dots, u_n)),
\end{displaymath}
interprets the rank of $(\emptyset, \corner{J_{u_n}[\mathcal{P}]},
\nu)$ in $<_{u_1,\dots,u_{n}}^{\mathcal{P}}$ as computed in the universe $L[\mathcal{P}]^{\coll(\omega, u_n)}$. In particular, we have by indiscernibility that 
\begin{displaymath}
  \rho^{L[\mathcal{P}]} (\nu, (u_1, \dots, u_n)) = \rho^{L_{u_{n+1}}[\mathcal{P}]} (\nu, (u_1, \dots, u_n)).  
\end{displaymath}

In this paper, by ``a countable iterate of $M_1$'', we mean an iterate
of $M_1$ via a hereditarily
countable stack of normal iteration trees according to the canonical
strategy of $M_1$. If $N$ is a countable iterate of $M_1$ and the
iteration map $\pi_{M_1,N}$ on the main branch exists, ``a
countable iterate of $N$'' means an iterate of $N$ via a hereditarily
countable stack of normal iteration trees according to the canonical
strategy of $N$. If $N$ is a countable iterate of $M_1$,
$\pi_{N,\infty}$ denotes the tail of the direct limit map from $N$ to
$M_{1,\infty}$. 
\begin{lemma}
  \label{lem:rho_n_fixed}
If $N$ is a countable iterate of $M_1$ and  $P$ is a further iterate of $N$ with iteration map $\pi_{NP}$, $\nu < \gamma^N_{\se{u_1,\dots,u_n}}$, then 
  \begin{displaymath}
    \rho^N (\nu, (u_1,\dots, u_{n})) = \rho^P ( \pi_{NP} (\nu), (u_1, \dots, u_{n})).
  \end{displaymath}
\end{lemma}
\begin{proof}
$\pi_{NP}$ moves the left hand side to the right hand side. So we automatically have the $\leq$ direction.

On the other hand, whenever $\alpha_1 < \dots < \alpha_{n}$ are countable Silver indiscernibles for $L[{N}^{-},P^{-},\mathcal{T}]$ where $\mathcal{T}$ is the countable tree leading from $N$ to $P$, 
$<^{P^{-}}_{\alpha_1,\dots,\alpha_{n}}$ embeds into $<^{N^{-}}_{\alpha_1,\dots,\alpha_{n}}$ via
\begin{displaymath}
( \seq{{\mathcal{U}}_i  }{i< k},\seq{\mathcal{Q}_i}{  i \leq k},  \eta ) \mapsto ({\mathcal{T}^{*}} \concat  \seq{{\mathcal{U}}_i  }{ i < k}, N | \alpha_n\concat \seq{\mathcal{Q}_i }{i \leq k},  \eta )
\end{displaymath}
where $\mathcal{T}^{*}$ is $\mathcal{T}$ construed as an iteration tree on $N | \alpha_n$. 
This embedding implies that
\begin{displaymath}
\rho^{N }(\nu, (\alpha_1,\dots,\alpha_n))  \geq\rho^{P}(\pi_{NP}(\nu), (\alpha_1,\dots,\alpha_n)) 
\end{displaymath}
and hence 
the $\geq$ direction of the lemma by indiscernibility. 
\end{proof}

\begin{definition}\label{def:P}
  $P_{u_1,\dots,u_n}$ is the set of $\alpha < u_{n+1}$ for which there is a countable iterate $N$ of $M_1$ and $\nu < \gamma^N_{\se{u_1,\dots,u_n}}$ such that
  \begin{displaymath}
\alpha = \rho^{N} (\nu, (u_1,\dots,u_n)).
\end{displaymath}

\end{definition}

Working in a model of the form $L[x]$ for some $x \in \mathbb{R}$, we say that
\begin{displaymath}
 \mathcal{Q} \text{ is } (\eta_1,\dots,\eta_n, u_{n+1}) \text{-iterable by $\rho$-value}
\end{displaymath}
iff $\mathcal{Q}$ is countable, suitable,  $\eta_1 < \dots < \eta_n < u_{n+1}$, $\beta < \gamma^{\mathcal{Q}}_{\se{\eta_1,\dots,\eta_n}}$ and
\begin{enumerate}
\item if $\mathcal{T}$ is a short tree on $\mathcal{Q}$ of length
  $\leq \omega_1$ with iteration map $\pi^{\mathcal{T}}$ on its main
  branch, then $\pi^{\mathcal{T}}(\eta_1, \dots, \eta_n) = (\eta_1,
  \dots, \eta_n)$ and for any $\beta <
  \gamma^{\mathcal{Q}}_{\se{\eta_1, \dots, \eta_n}}$, $\pi^{\mathcal{T}}(\rho^{L_{u_{n+1}}[\mathcal{Q}]} ( \beta,
  (\eta_1,\dots,\eta_n)) ) = \rho^{L_{u_{n+1}}[\mathcal{Q}]} ( \beta,
  (\eta_1,\dots,\eta_n))$.
\item if $\mathcal{T}$ is a maximal tree on $\mathcal{Q}$ of length
  $\leq \omega_1$, then there is a branch
  $b \in L[x]^{\coll(\omega,u_{n+1})}$ such that
  $u_{n+1}$ is contained in the wellfounded
  part of $\mathcal{M}_b^{\mathcal{T}}$,
  $\pi_b^{\mathcal{T}}(\eta_1,\dots,\eta_n, u_{n+1}) =
  (\eta_1,\dots,\eta_n, u_{n+1})$ and for any $\beta <
  \gamma^{\mathcal{Q}}_{\se{\eta_1, \dots, \eta_n}}$, 
  $\pi_b^{\mathcal{T}}(\rho^{L_{u_{n+1}}[\mathcal{Q}]} ( \beta,
  (\eta_1,\dots,\eta_n)) ) = \rho^{L_{u_{n+1}}[\mathcal{Q}]} ( \beta,
  (\eta_1,\dots,\eta_n))$.
\end{enumerate}
By $\Sigma^1_1$-absoluteness and Lemma~\ref{lem:rho_n_fixed}, for any
countable iterate $N$ of $M_1$, if $x$ is a real and $N^{-} \in
HC^{L[x]}$,
\begin{displaymath}
  L[x] \models N^{-} \text{ is $(u_1, \dots, u_n, u_{n+1})$-iterable by
  $\rho$-value}. 
\end{displaymath}

We will show that $P_{u_1, \dots, u_n} \in
L_{u_{\omega}}[\mathcal{O}_{\Sigma^1_3}]$ by estimating the
complexity. Recall in \cite{hjorth_boundedness_lemma} the definition of the pointclass
$\Gamma_{1,n}$.  $A \subseteq \mathbb{R}$ is in $\Gamma_{1,n}$ iff there is a formula $\varphi$ such that
\begin{displaymath}
  A = \set{x}{ L[x] \models \varphi (x, u_1, \dots, u_n)}. 
\end{displaymath}
We have by Martin \cite{martin_largest}
\begin{displaymath}
  \game ( \omega n \text{-} \Pi^1_1 )\subseteq \Gamma_{1,n} \subseteq \game ( \omega (n+1) \text{-} \Pi^1_1 ).
\end{displaymath}
We now allow the pointclass to act on ordinals as well. $A \subseteq \mathbb{R} \times u_{\omega}$ is said to be in $\Gamma_{1,n}$ iff there is a formula $\varphi$ such that
\begin{displaymath}
  A = \set{(x, \alpha)}{ L[x] \models \varphi (x, \alpha, u_1, \dots, u_n)}. 
\end{displaymath}
If $C \subseteq \mathbb{R}$, $G(C)$ is the infinite game on $\omega$
in which two players collabrate to produce a real $x$ and I wins iff
$x \in C$. If $A \subseteq \mathbb{R} \times u_{\omega}$, $B$ is the
set of $\alpha < u_{\omega}$ such that I has a winning strategy in
$G(A_\alpha)$, where $A_\alpha = \{ \alpha < u_{\omega} : (x, \alpha)
\in A\}$. 
Naturally, $B \subseteq u_{\omega}$ is said to be in $\game \Gamma_{1,n}$ iff there is $A \subseteq \mathbb{R} \times u_{\omega}$ in $\Gamma_{1,n}$ such that $B = \game A$. 

\begin{lemma}
    \label{lem:P_n_complexity}
    $P_{u_1,\dots,u_n}$ is $\game  \Gamma_{1,n+1} $. 
\end{lemma}
\begin{proof}
  We claim that for $\alpha < u_{n+1}$, 
  \begin{displaymath}
  \alpha \in P_{u_1,\dots,u_n}
  \end{displaymath}
iff for a cone of $x$, $L[x]$ satisfies that there is $(\mathcal{Q},\beta)$ such that $\mathcal{Q}$ is $(u_1,\dots,u_n, u_{n+1})$-iterable by $\rho$-value and 
\begin{displaymath}
  \alpha= \rho^{L_{u_{n+1}}[\mathcal{Q}]}(\beta,u_1,\dots,u_n)
\end{displaymath}

$\Rightarrow$: If $\alpha \in P_{u_1, \dots, u_n}$, then there is a countable iterate $N$ of $M_1$ and $\nu < \gamma^N_{\se{u_1, \dots, u_n}}$ such that $\alpha = \rho^N(\nu, (u_1, \ldots, u_n))$. For any $x$ satisfying $N^{-} \in HC^{L[x]}$, by Lemma~\ref{lem:rho_n_fixed}, $N^{-}$ is $(u_1, \dots,u_n, u_{n+1})$-iterable by $\rho$-value in $L[x]$. This verifies the $\Rightarrow$ direction. 

 $\Leftarrow$: Suppose $\alpha < u_{n+1}$
 and for a cone of $x \geq_T w$, $L[x]$ satisfies the above
 statement. Pick such an $x \geq_T M_1^{\#}$. Pick a witness
 $(\mathcal{Q}, \beta )\in HC^{L[x]}$ such that $\mathcal{Q}$ is
 $(u_1,\dots,u_n, u_{n+1})$-iterable by $\rho$-value in $L[x]$ and $\rho^{L_{u_{n+1}}[\mathcal{Q}]}(\beta, (u_1,\ldots, u_n)) = \alpha$. 
Working in $L[x]$, there is a pseudo-comparison $(\mathcal{T}, \mathcal{U})$ of $\mathcal{Q}$ and $M_1^{-}$ of length $\leq \omega_1^{L[x]}$, leading to a common pseudo-iterate $\mathcal{R}$ with $\delta^{\mathcal{R}} \leq \omega_1^{L[x]}$. Let $b$ be a branch choice for $\mathcal{T}$ in $L[x]^{\coll(\omega,u_{n+1})}$ such that $\pi_b^{\mathcal{T}}(u_1,\dots,u_{n+1}) = (u_1,\dots,u_{n+1})$ and $\pi_b^{\mathcal{T}}(\rho^{L_{u_{n+1}}[\mathcal{Q}]} ( \beta, (u_1,\dots,u_n)) ) = \rho^{L_{u_{n+1}}[\mathcal{Q}]} ( \beta, (u_1,\dots,u_n)) $. Then $\pi_b^{\mathcal{T}}(\beta) < \gamma^{L[\mathcal{R}]}_{\se{u_1,\dots,u_n}}$. But $L[\mathcal{R}]$ is a genuine iterate of $M_1$. 
$L[\mathcal{R}]$ and $\pi^{\mathcal{T}}_b(\beta)$ witnesses that $\alpha \in P_{u_1,\dots,u_n}$. 
\end{proof}

Zhu
in \cite{sharpI} proves the equality of pointclasses
\begin{displaymath}
  \game^2 (< \! \omega^2 \text{-}\Pi^1_1) =  < \! u_{\omega} \text{-}\Pi^1_3
\end{displaymath}
on subsets of $\mathbb{R}$. 
We produce a variant of this equality by allowing ordinal parameters. 
Recall the relevant definitions. If $\alpha$ is an ordinal and $A \subseteq \alpha \times X$, then put
\begin{displaymath}
x \in \Diff A \eqiv  \exists i < \alpha~(i \text{ is odd}\wedge \forall j < i ( (j,x) \in A) \wedge (i,x) \notin A).
\end{displaymath}
If $B$ is either a subset of $\mathbb{R}$ or a subset of $u_{\omega}$, $\alpha \leq u_{\omega}$, then  $B$ is said to be $\alpha$-$\Pi^1_3$ iff there is a $\Pi^1_3$ set $A$, a subset of either $\alpha \times \mathbb{R}$ or $\alpha \times u_{\omega}$ respectively,  such that $B = \Diff A$. 
The variant of this equality of
pointclasses on subsets of $u_{\omega}$ is:

\begin{lemma}
  \label{lem:pointclass_transfer}
Assume $\boldDelta{2}$-determinacy.  Let $B \subseteq u_{n+1}$ be
$\game \Gamma_{1,n}$. Then $B$ is $u_{n+2}\text{-}\Pi^1_3$. 
\end{lemma}
\begin{proof}
We follow closely the proof in \cite{sharpI}. Suppose that $B = \game
A$, $A \subseteq \mathbb{R} \times u_{n+1}$, and $A$ is
$\Gamma_{1,n}$. Fix a formula $\varphi$ such that
\begin{displaymath}
  (x,\alpha)\in A \eqiv L[x] \models \varphi( x, \alpha, u_1, \dots,
  u_n). 
\end{displaymath}

For countable ordinals $\xi , \eta_1, \dots, \eta_n ,\eta$ such that
$\max(\xi, \eta_1, \dots, \eta_n) < \eta$, we say that $M$
is a Kechris-Woodin non-determined set with respect to $(\xi,
\eta_1, \dots, \eta_n, \eta)$ iff
\begin{enumerate}
\item $M$ is a countable subset of $\mathbb{R}$;
\item $M$ is closed under join and Turing reducibility;
\item $\forall \sigma \in M~ \exists v \in M ~ L_{\eta}[\sigma \otimes
  v] \models \neg \varphi (\sigma \otimes v, \xi, \eta_1, \dots, \eta_n)$;
\item $\forall \sigma \in M ~ \exists v \in M  ~ L_{\eta}[v   \otimes
  \sigma ] \models \varphi (v \otimes \sigma, \xi, \eta_1, \dots, \eta_n)$.
\end{enumerate}
In clause 3, ``$\forall \sigma \in M$'' is quantifying over all strategies $\sigma$ for Player I that is coded in some member of $M$; $\sigma * v$ is Player I's response to $v$ according to $\sigma$, and $\sigma \otimes v = (\sigma * v) \oplus v$ is the combined infinite run. Similarly for clause 4, roles between two players being exchanged. 
Say that $z$ is $(\xi,\eta_1, \dots, \eta_n, \eta)$-stable iff $z$ is
not contained in any Kechris-Woodin non-determined set with respect to
$(\xi,\eta_1, \dots, \eta_n, \eta)$. $z$ is stable iff $z$ is
$(\xi,\eta_1, \dots, \eta_n, \eta)$-stable for all $\xi,\eta_1, \dots,
\eta_n, \eta$ such that $\max(\xi, \eta_1, \dots, \eta_n) < \eta <
\omega_1$. Being stable is a
$\Pi^1_2$-property. 
The following claim is extracted from the proof of the Kechris-Woodin
determinacy transfer theorem in \cite{KW_det_transfer} that
\begin{displaymath}
  \Delta^1_2 \text{-Determinacy} \Rightarrow
  \game(<\!\omega^2\text{-$\Pi^1_1$})\text{-Determinacy}. 
\end{displaymath}
\begin{claim}
  \label{claim:kechris-Woodin}
There is a stable real.
\end{claim}
\begin{proof}
  Suppose otherwise. The set of $(z, y) \in \mathbb{R}$ such that for
  some $(\xi, \eta_1, \dots, \eta_n, \eta)$, $y$
  codes a Kechris-Woodin non-determined set $M_y$ with
  respect to $(\xi, \eta_1, \dots, \eta_n, \eta)$ and such that $z \in
  M_y$ is $\Sigma^1_2$. By $\Sigma^1_2$-uniformization, this set is
  uniformized by a $\Sigma^1_2$ function $F$. $F$ is total by
  assumption. 
Thus, $F$ is $\Delta^1_2$. 
Denote the Kechris-Woodin non-determined set coded in $F(z)$ by
  $F^{*}(z)$.  
For any $z \in \mathbb{R}$,
  define $(\xi^z, \eta_1^z, \dots, \eta_n^z, \eta^z)$ as the lexicographically least
  tuple such that
  \begin{multline*}
    \exists y \leq_T z ( z \in F^{*}(y) \wedge F^{*}(y) \text{ is a
      Kechris-Woodin non-determined set}\\
    \text{with respect to } (\xi^z, \eta_1^z, \dots, \eta_n^z, \eta^z
    )). 
  \end{multline*}
Consider the game in which I produces $z_0, x_0$, II produces $z_1,
x_1$. Let $z = z_0 \oplus z_1$ and $x = x_0 \oplus x_1$. Then I wins iff
\begin{displaymath}
  L_{\eta^z} [ x] \models \varphi (x, \xi^z, \eta_1^z, \dots,
  \eta_n^z). 
\end{displaymath}
This game is $\Delta^1_2$, hence determined. Suppose with loss of
generality that I has a winning strategy $\bar{\sigma}$. 

We have $z \equiv_T z' \to (\xi^z, \eta_1^z, \dots, \eta_n^z, \eta^z)
= (\xi^{z'}, \eta_1^{z'}, \dots, \eta_n^{z'}, \eta^{z'})$. Since the
ordinals are wellfounded, we have
\begin{displaymath}
  \forall z ~\exists z' \geq_T z ~ \forall z'' \geq_T z' ~ (\xi^{z''},
  \eta_1^{z''}, \dots, \eta_n^{z''}, \eta^{z''} ) \geq_{\text{lex}}  (\xi^{z'},
  \eta_1^{z'}, \dots, \eta_n^{z'}, \eta^{z'} )
\end{displaymath}
 By $\Delta^1_2$-Turing determinacy,
we find $w_0 \geq_T \bar{\sigma}$ such that 
\begin{displaymath}
  \forall z \geq_T w_0 ~(\xi^z,
  \eta_1^z, \dots, \eta_n^z, \eta^z ) \geq_{\text{lex}}  (\xi^{w_0},
  \eta_1^{w_0}, \dots, \eta_n^{w_0}, \eta^{w_0} )
\end{displaymath}
Let $\sigma_0$
be a strategy for I such that $ \bar{\sigma} * (w_0, x_1) = (z_0, \sigma_0
* x_1)$ for some $z_0$. Of course, $\sigma_0 \leq _T w_0$.

Pick a
real $z \leq_T w_0$ such that $w_0 \in F^{*}(z)$ and $F^{*}(z)$
is a Kechris-Woodin non-determined set with respect to $ (\xi^{w_0},
  \eta_1^{w_0}, \dots, \eta_n^{w_0}, \eta^{w_0} )$. 
However, we shall produce a contradiction to clause 3 of the
definition of Kechris-Woodin non-determined set by proving that 
\begin{displaymath}
  \forall v \in F^{*}(z) ~
  L_{\eta^{w_0}}[\sigma_0 \otimes v] \models  \varphi (\sigma_0 \otimes v,
  \xi^{w_0}, \eta_1^{w_0}, \dots, \eta_n^{w_0}). 
\end{displaymath}
Suppose that $v \in F^{*}(z)$. Let $\bar{\sigma} * (w_0, v) = (z_0,
x_0)$. Then $x_0 = \sigma_0 * v$. Let $z' = w_0 \oplus z_0$. Since $\bar{\sigma}$ is winning, we have
\begin{displaymath}
  L_{\eta^{z'}}[ \sigma_0 \otimes v] \models \varphi (\sigma_0 \otimes
  v, \xi^{z'}, \eta_1^{z'}, \dots, \eta_n^{z'}).
\end{displaymath}
Thus, it suffices to show that $(\xi^{z'}, \eta_1^{z'}, \dots, \eta_n^{z'},
\eta^{z'}) = (\xi^{w_0}, \eta_1^{w_0}, \dots, \eta_n^{w_0},
\eta^{w_0})$. We have the $\geq_{\text{lex}}$ direction because $z'
\geq_T w_0$. To see the $\leq_{\text{lex}}$ direction, just note that
$F^{*}(z)$ contains $z'$ and is already a Kechris-Woodin non-determined set with respect
to $ (\xi^{w_0},
  \eta_1^{w_0}, \dots, \eta_n^{w_0}, \eta^{w_0} )$. This finishes the
  proof of Claim~\ref{claim:kechris-Woodin}.
\end{proof}

Let $<^{\xi, \eta_1, \dots, \eta_n, \eta}$ be the following
wellfounded relation on the set of $z$ which is $(\xi, \eta_1, \dots, \eta_n, \eta)$-stable:
\begin{align*}
  z' <^{\xi, \eta_1, \dots, \eta_n, \eta} z \eqiv  &  z \text{ is } (\xi, \eta_1, \dots, \eta_n, \eta)\text{-stable} \wedge z \leq_T z'  \wedge \\
  & \forall \sigma \leq_T z ~\exists v \leq_T z' ~ L_{\eta}[\sigma
    \otimes v] \models \neg \varphi (\sigma \otimes v, \xi, \eta_1,
    \dots, \eta_n)  \wedge \\
  &\forall \sigma \leq_T z ~\exists v \leq_T z' ~ L_{\eta}[v \otimes
    \sigma] \models  \varphi (v \otimes \sigma, \xi, \eta_1, \dots, \eta_n).
\end{align*}
The reason why $<^{\xi, \eta_1, \dots, \eta_n, \eta}$ is wellfounded is because otherwise, there would exist $(z_n:n<\omega)$ such that $z_0$ is $(\xi, \eta_1, \dots, \eta_n,\eta)$-stable and $z_{n+1}<^{\xi, \eta_1, \dots, \eta_n, \eta} z_n$ for each $n$, and thus one can build a Kechris-Woodin non-determined set
\begin{displaymath}
  \set{x}{\exists n (x \leq_T z_n)}.
\end{displaymath}
that contains $z_0$, a contradiction.  
If $z$ is $(\xi, \eta_1, \dots, \eta_n, \eta)$-stable, then $<^{\xi,
  \eta_1, \dots, \eta_n, \eta} \res \set{z'}{z' <^{\xi, \eta_1, \dots,
    \eta_n, \eta} z}$ is a $\boldsigma{1}$ wellfounded relation in the
code of $(\xi, \eta_1, \dots, \eta_n, \eta)$, hence has rank $<
\omega_1$ by Kunen-Martin. If $z$ is stable, let $f^z$ be the function that sends $(\xi, \eta_1, \dots, \eta_n, \eta)$ to the rank of $z$ in $<^{\xi, \eta_1, \dots,
    \eta_n, \eta} $. By Shoenfield absoluteness, there is a Skolem term $\tau$ in the language of set theory such that for all $z \in \mathbb{R}$, if $z$ is stable, then
\begin{displaymath}
  f^z(\xi, \eta_1, \dots, \eta_n, \eta) = \tau^{L[z]}(z, \xi, \eta_1, \dots, \eta_n, \eta).
\end{displaymath}
Let
\begin{displaymath}
  \beta^z_{\alpha} = \tau^{L[z]}(z, \alpha, u_1, \dots, u_n,  u_{n+1}).
\end{displaymath}
The function
\begin{displaymath}
  z \mapsto \beta^z_{\alpha}
\end{displaymath}
is $\Delta^1_3(\alpha)$ in the sharp codes. We say that $z$ is
$\alpha$-ultrastable iff $z$ is stable and $\beta^z_{ \alpha} =
\min\set{\beta^{z'}_{ \alpha}}{z' \text{ is stable}}$. The same
argument in \cite{sharpI} shows that:
\begin{claim}
\label{claim:ultrastable}
  If $z$ is $\alpha$-ultrastable, then $z$ computes a
  winning strategy in $G(A)$ for one  of the two players.
\end{claim}
\begin{proof}
  Suppose otherwise. Let $w \in \WO$ such that $\sharpcode{w} =
  \alpha$.  
For each $\sigma \leq_T z$ for either of the two
  Players in $G(A)$, find a defeat $y_{\sigma}$ of $\sigma$. Let $z'$
  be Turing above $w \oplus z$ and above $y_{\sigma}$ for any $\sigma \leq_T
  z$. 
By indiscernibility, whenever $(\xi, \eta_1, \dots, \eta_n, \eta)$ are
$L[z']$-indiscernibles and $\xi < \eta_i \eqiv \alpha < u_i$ and $\xi
= \eta_i \eqiv \alpha = u_i$ for any $1 \leq i \leq n$, we have
\begin{displaymath}
  z' < ^{\xi, \eta_1, \dots, \eta_n, \eta} z
\end{displaymath}
and hence 
\begin{displaymath}
  f^{z'}(\xi, \eta_1, \dots, \eta_n, \eta) < f^z(\xi, \eta_1, \dots,
  \eta_n, \eta).
\end{displaymath}
Therefore, $\beta^{z'}_{\alpha} < \beta^z_{\alpha}$, contradicting to
$\alpha$-ultrastableness of $z$.
\end{proof}
We then let
\begin{displaymath}
  (\alpha, \gamma, z) \in C
\end{displaymath}
iff $z$ is $\alpha$-stable and $\beta^z_{\alpha} = \gamma$. $C$ is
$\Delta^1_3$. Then
\begin{displaymath}
  \alpha \in B
\end{displaymath}
iff
\begin{multline*}
  \text{if $\gamma_0$ is the smallest such that $\exists z~(\alpha,
    \gamma_0, z) \in C$,}\\
\text{then }
\forall z ((\alpha, \gamma_0, z) \in C \to \exists \sigma \leq_T z ~\forall v ~(\alpha, \sigma \otimes v )\in A). 
\end{multline*}
Thus,  $B$ is $u_{n+2}\text{-}\Pi^1_3$ by the following
definition:
\begin{displaymath}
  B = \Diff(E),
\end{displaymath}
where $E \subseteq u_{n+2} \times u_{n+1}$ is given by: $(2\gamma,\alpha) \in E$ iff $\neg \exists z~(\alpha, \gamma_0,z) \in C$, and $(2\gamma+1, \alpha) \in E$ iff $\forall z ((\alpha, \gamma_0, z) \in C \to \exists \sigma \leq_T z ~\forall v ~(\alpha, \sigma \otimes v )\in A)
$.
\end{proof}

By Lemmas~\ref{lem:P_n_complexity}-\ref{lem:pointclass_transfer}, $P_{u_1,\dots,u_n}$ is $u_{n+3}$-$\Pi^1_3$. 
By Theorem~\ref{thm:becker_kechris_martin}, $P_{u_1,\dots,u_n} \in L_{\kappa_3}[T_2]$. 
Let
\begin{displaymath}
  f_n : \delta_n\to P_{u_1,\dots,u_n}
\end{displaymath}
be the order preserving enumeration of $P_{u_1,\dots,u_n}$. Then $f_n \in L_{\kappa_3}[T_2]$ and hence by Theorem~\ref{thm:becker_kechris_martin}, there is $\mu_n < u_{n+2}$ such that $f_n$ is $\Delta^1_3(\mu_n)$. We fix this $\mu_n$ and fix a $\Sigma^1_3$ set
\begin{displaymath}
  B_n \subseteq u_{n+2} \times (u_{n+1} \times u_{n+1})
\end{displaymath}
such that
\begin{displaymath}
  f_n (\alpha) = \beta \eqiv (\mu_n, ( \alpha, \beta)) \in B_n 
\end{displaymath} 
The role of $f_n$ is to compute $\pi_{N,\infty} (\alpha)$ for a
countable iterate $N$ of $M_1$ and $\alpha < \gamma^N_{\se{u_1, \dots, u_n}}$:
\begin{lemma}
  \label{lem:f-n-use}
Suppose that $N$ is a countable iterate of $M_1$ and $\alpha < \gamma^N_{\se{u_1, \dots, u_n}}$. Then
\begin{displaymath}
  f_n ( \pi_{N,\infty} (\alpha)) = \rho^N (\alpha, \se{u_1, \dots, u_n}). 
\end{displaymath}
\end{lemma}
\begin{proof}
  Define  a map $\sigma$ sending $\pi_{N,\infty}(\alpha)$ to $\rho^N
  (\alpha, \se{u_1, \dots, u_n})$ for $N$ a countable iterate of $M_1$
  and $\alpha < \gamma^N_{\se{u_1, \dots, u_n}}$. By comparison and
  Lemma~\ref{lem:rho_n_fixed}, $\pi$ is well defined and order
  preserving. By definition, the range of $\sigma$ is exactly $P_{u_1,\dots,u_n}$. Therefore, $\sigma = f_n$. 
\end{proof}

Fix a $\Sigma^1_3$-formula
\begin{displaymath}
\varphi_{B_n}
\end{displaymath}
such that $\varphi_{B_n}(w, z_1,z_2)$ iff $w,z_1,z_2 \in \WO_{n+2}$ and $(\sharpcode{w}, (\sharpcode{z_1}, \sharpcode{z_2})) \in B_n$. 
Inside a model of the form $L[x]$, let
  \begin{align*}
    f_{n,u_1,\dots, u_{n+1}}^{\mu_n} =
    \{ (\alpha, \beta): ~& \exists w , z_1, z_2 ~( \varphi_{B_n}(w, z_1, z_2)  \wedge \\
    &\sharpcode{w} = \mu_n \wedge \sharpcode{z_1} = \alpha \wedge \sharpcode{z_2} = \beta \\
 &     \text{using $u_1, \dots, u_{n+1}$ to evaluate $\sharpcode{w}, \sharpcode{z_1}, \sharpcode{z_2}$} ). 
  \end{align*}
be the partial function defined from $u_1,\dots,u_{n+1},\mu_n$. By upward $\Sigma^1_3$ absoluteness, for any real $x$,
\begin{displaymath}
  (f^{\mu_n}_{n,u_1,\dots,u_{n+1}})^{L[x]} \subseteq f_n
\end{displaymath}
and for any $y \geq_T x$, 
\begin{displaymath}
   (f^{\mu_n}_{n,u_1,\dots,u_{n+1}})^{L[x]} \subseteq  (f^{\mu_n}_{n,u_1,\dots,u_{n+1}})^{L[y]} .
\end{displaymath}
Hence,
\begin{displaymath}
  f_n = \bigcup \set{  (f^{\mu_n}_{n,u_1,\dots,u_{n+1}})^{L[x]}  }{x \in \mathbb{R}}.
\end{displaymath}


A countable iterate $N$ of $M_1$ is said to be $\alpha$-stable iff for
any further countable iterate $P$ of $N$ with iteration map $\pi_{NP}$, $\pi_{NP}(\alpha) = \alpha$. If $s$ is a finite set of ordinals, $N$ is $s$-stable iff $N$ is $\alpha$-stable for any $\alpha \in s$.  
The iterability of $M_1$ implies that for any finite set of ordinals
$s$, there is a countable iterate $N$ of $M_1$ which is $s$-stable. 
Let
\begin{displaymath}
G(\alpha) = \pi_{N,\infty}(\alpha)
\end{displaymath}
where $N$ is $\alpha$-stable. 
\begin{lemma}
  \label{lem:LT2_in_M1infty}
Assume $\boldDelta{2}$-determinacy.  Suppose that $\varphi$ is a $\Sigma^1_3$ formula. Then for any $\alpha < \delta_n$,
  \begin{displaymath}
    \exists v (v \in \WO_{\omega} \wedge \sharpcode{v} = \alpha \wedge \varphi(v))
  \end{displaymath}
iff
\begin{multline*}
  M_{1,\infty} ^{\coll(\omega, \delta_{1,\infty})}\models 
\exists \corner{\gcode{\tau},a^{\#}}~ (\varphi(\corner{\gcode{\tau},a^{\#}}) \wedge \\
 f^{G(\mu_n)}_{n,G(u_1),\dots,G(u_{n+1})} (\tau^{L[a]}(a, G(u_1),\dots,G(u_n))) =\rho^{M_{1,\infty} | G(u_{n+1})}(\alpha, (G(u_1),\dots, G(u_n))) .
\end{multline*}
\end{lemma}
\begin{proof}
  Let $N$ be $\se{\alpha, \mu_n}$-stable such that  
$\pi_{N,\infty}(\bar{\alpha}) =\alpha$. Then by Lemma~\ref{lem:f-n-use}, 
  \begin{displaymath}
    \rho^{N | u_{n+1}} (\bar{\alpha}, (u_1,\dots,u_n)) = f_n (\alpha).
  \end{displaymath}
By elementarity, it suffices to show that 
  \begin{displaymath}
    \exists v (v \in \WO_{\omega} \wedge \sharpcode{v} = \alpha \wedge \varphi(v))
  \end{displaymath}
iff
\begin{displaymath}
   N ^{\coll(\omega, \delta^N)} \models 
\exists v~ (   
f^{\mu_n}_{n,u_1,\dots,u_{n+1}} ( \sharpcode{v}) = \rho^{N | u_{n+1}} (\bar{\alpha}, (u_1, \dots, u_n) )
\wedge \varphi(v)  ),
\end{displaymath}
or in other words, iff 
\begin{displaymath}
   N ^{\coll(\omega, \delta^N)} \models 
\exists v~ (  
f^{\mu_n}_{n,u_1,\dots,u_{n+1}} ( \sharpcode{v}) =f_n(\alpha))
\wedge 
 \varphi(v)  ). 
\end{displaymath}

$\Leftarrow$: We have by assumption a $\coll(\omega, \delta^N)$-generic filter $g$ over $N$ and $v_0 \in N[g]$ such that 
\begin{displaymath}
  N[g] \models f^{\mu_n}_{n, u_1, \ldots, u_{n+1}} (\sharpcode{v_0} )= f_n(\alpha)  \wedge \varphi(v_0)  .
\end{displaymath}
By upward $\Sigma^1_3$ absoluteness, $\varphi(v_0)$ holds in $V$ and $f_n (\sharpcode{v_0}) = f_n (\alpha)$. Therefore, $\sharpcode{v_0} = \alpha$. $v_0$ verifies the existence quantifier in the conclusion of the $\Leftarrow$ direction. 

$\Rightarrow$: Let $\varphi (v)$ be $\exists y \theta (v,y)$, where $\theta$ is $\Pi^1_2$. Let $\sharpcode{v_0} = \alpha$ and $y_0$ be such that $\theta(v_0,y_0)$. Let $v_1$ be such that
\begin{displaymath}
    (f^{\mu_n}_{n,u_1,\dots,u_{n+1}})^{L[v_1]} (\alpha) = f_n (\alpha).  
\end{displaymath}
 Iterate $N$ to some $P$ so that $\corner{v_0,v_1,y_0}$ is $\mathbb{B}^P$-generic over $P$. Let $v_2$ be a real such that $L[v_2]$ is a $\coll(\omega, \delta^P)$-extension of $P$ and $\corner{v_0,v_1,y_0} \in L[v_2]$. Then
 \begin{displaymath}
   L[v_2] \models  
f^{\mu_n}_{n,u_1,\dots,u_{n+1}} ( \sharpcode{v_0}) = f_n(\alpha) \wedge \varphi(v_0).   
 \end{displaymath}
Thus, 
 \begin{displaymath}
  P ^{\coll(\omega, \delta^P)}\models  \exists v~
(f^{\mu_n}_{n,u_1,\dots,u_{n+1}} ( \sharpcode{v}) = \rho^{P | u_{n+1}} (\pi_{NP}(\bar{\alpha}), (u_1, \dots, u_n)) \wedge \varphi(v)).   
 \end{displaymath}
And pull it back via elementarity.
\end{proof}

\begin{theorem}
  \label{thm:LT_2_equal_to_M_1_infty}
Assume $\boldDelta{2}$-determinacy. Then  $  L_{u_{\omega} } [\mathcal{O}_{\Sigma^1_3}]$ and $ M_{1,\infty} | u_{\omega}$ have the same universe.
\end{theorem}
\begin{proof}
  The universe of $  L_{u_{\omega} } [\mathcal{O}_{\Sigma^1_3}]$ is a subset of  that of $ M_{1,\infty} | u_{\omega}$: By Lemma~\ref{lem:LT2_in_M1infty} and Hjorth \cite{hjorth_boundedness_lemma} that $\sup_{n<\omega} \delta_n = u_{\omega}$.

 The universe of $ M_{1,\infty} | u_{\omega}$ is a subset of that of  $  L_{u_{\omega} } [\mathcal{O}_{\Sigma^1_3}]$: Suppose $a \subseteq u_n$ is in $M_{1,\infty}$. Let $\alpha_0 < u_k$, $n<k<\omega$ and $\varphi$ be such that
 \begin{displaymath}
   a = \set{\alpha < u_n}{ M_{1,\infty} \models \varphi(\alpha,\alpha_0)}.
 \end{displaymath}
We show that $a$ has a $\game \Gamma_{1,k+1}(\alpha_0)$ definition:
\begin{displaymath}
  \alpha \in a
\end{displaymath}
iff for a cone of $x$, 
\begin{multline*}
  L[x] \models \exists \mathcal{Q}, \beta ,\beta_0 \in HC\\
  ( \mathcal{Q} \text{ is $(u_1, \ldots, u_{k+1})$-iterable by $\rho$-value} \wedge \\
f^{\mu_k}_{k,u_1,\dots, u_{k+1}} (\alpha) =   \rho^{L_{u_{k+1}}[\mathcal{Q}]} (\beta, u_1,\dots, u_k) \wedge\\
f^{\mu_k}_{k,u_1,\dots, u_{k+1}} (\alpha_0) =   \rho^{L_{u_{k+1}}[\mathcal{Q}]} (\beta_0, u_1,\dots, u_k) \wedge\\
L[\mathcal{Q}] \models \varphi (  \beta, \beta_0 )).
\end{multline*}

$\Rightarrow$: Suppose that $\alpha \in a$. Iterate $M_1$ to $N$ via a
countable iteration such that for some $\beta, \beta_0 \in N$,
$\pi_{N,\infty} (\beta, \beta_0) = (\alpha, \alpha_0)$. Let $x_0$ be a real coding $N^{-}$. 
Then for any $x \geq_T x_0$, 
$N^{-}$
is $(u_1, \dots, u_{k+1})$-iterable by $\rho$-value in $L[x]$. 
Let $x_1\geq_T x_0$ be a real such that
\begin{displaymath}
  L[x_1] \models f^{\mu_k}_{k,u_1,\dots, u_{k+1}} (\alpha) = f_k (\alpha) \wedge f^{\mu_k}_{k,u_1,\dots, u_{k+1}} (\alpha_0) = f_k (\alpha_0). 
\end{displaymath}
Then for any $x \geq_T x_1$, 
\begin{displaymath}
  L[x] \models f^{\mu_k}_{k,u_1,\dots, u_{k+1}} (\alpha) = f_k (\alpha) \wedge f^{\mu_k}_{k,u_1,\dots, u_{k+1}} (\alpha_0) = f_k (\alpha_0). 
\end{displaymath}
 By Lemma~\ref{lem:f-n-use}, 
\begin{multline*}
  L[x] \models (
f^{\mu_k}_{k,u_1,\dots, u_{k+1}} (\alpha) =   \rho^{L_{u_{k+1}}[N^{-}]} (\beta, u_1,\dots, u_k) \wedge\\
f^{\mu_k}_{k,u_1,\dots, u_{k+1}} (\alpha_0) =   \rho^{L_{u_{k+1}}[N^{-}]} (\beta_0, u_1,\dots, u_k) )
\end{multline*}
The assumption $\alpha \in a$ implies that $M_{1, \infty } \models
\varphi (\alpha, \alpha_0)$. By elementarity, $N \models
\varphi(\beta, \beta_0)$. $(N^{-}, \beta, \beta_0)$ plays the role of
$(\mathcal{Q}, \beta, \beta_0)$ in the existential quantifier of the
statement in $L[x]$.

$\Leftarrow$: Let $x_0 \geq M_1^{\#}$ and let  $\mathcal{Q}, \beta, \beta_0 \in HC^{L[x_0]}$ such that  
\begin{multline*}
  L[x_0] \models \mathcal{Q} \text{ is $(u_1, \ldots, u_{k+1})$-iterable by $\rho$-value} \wedge \\
f^{\mu_k}_{k,u_1,\dots, u_{k+1}} (\alpha) =   \rho^{L_{u_{k+1}}[\mathcal{Q}]} (\beta, u_1,\dots, u_k) \wedge\\
f^{\mu_k}_{k,u_1,\dots, u_{k+1}} (\alpha_0) =   \rho^{L_{u_{k+1}}[\mathcal{Q}]} (\beta_0, u_1,\dots, u_k) \wedge\\
L[\mathcal{Q}] \models \varphi (  \beta, \beta_0 ).
\end{multline*}
Thus,
\begin{displaymath}
  f_k(\alpha) = \rho^{L_{u_{k+1}}[\mathcal{Q}]} (\beta, u_1,\dots,
  u_k) \wedge f_k (\alpha_0) =   \rho^{L_{u_{k+1}}[\mathcal{Q}]} (\beta_0, u_1,\dots, u_k) .
\end{displaymath}
Pseudo-compare $\mathcal{Q}$ with $M_1^{\#}$ in $L[x_0]$, leading to a
common pseudo-normal-iterate $\mathcal{R}$ with $\delta^{\mathcal{R}}
\leq \omega_1^{L[x_0]}$. In $L[x_0]^{\coll(\omega, u_{k+1})}$, there
is a branch choice in the pseudo-normal-iteration on the
$\mathcal{Q}$-side whose branch map fixes
\begin{displaymath}
(u_1, \dots, u_{k+1},
\rho^{L_{u_{k+1}}[\mathcal{Q}]}(\beta, u_1, \dots, u_k),
\rho^{L_{u_{k+1}}[\mathcal{Q}]}(\beta_0, u_1, \dots, u_k).
\end{displaymath}
Let $(\gamma, \gamma_0)$
be the image of $(\beta, \beta_0)$ under this branch map. Then
\begin{displaymath}
  f_k(\alpha) = \rho^{L_{u_{k+1}}[\mathcal{R}]}(\gamma, u_1, \dots, u_k)
  \wedge f_k (\alpha_0) =   \rho^{L_{u_{k+1}}[\mathcal{R}]} (\gamma_0, u_1,\dots, u_k) 
\end{displaymath}
Since $L[\mathcal{Q}] \models \varphi(\beta, \beta_0)$, we have $L_{u_{k+1}}[\mathcal{Q}] \models \varphi (\beta, \beta_0)$ by
indiscernibility. Thus, by elementarity, 
\begin{displaymath}
  L_{u_{k+1}}[\mathcal{R}] \models \varphi (\gamma, \gamma_0).
\end{displaymath}
and by indiscernibility again, 
\begin{displaymath}
  L[\mathcal{R}] \models \varphi (\gamma, \gamma_0).
\end{displaymath}
But $L[\mathcal{R}]$ is a genuine iterate of $M_1$. Thus,
\begin{displaymath}
  M_{1, \infty} \models \varphi ( \pi_{L[\mathcal{R}], \infty}
  (\gamma), \pi_{L[\mathcal{R}], \infty} (\gamma_0 )). 
\end{displaymath}
By Lemma~\ref{lem:f-n-use}, $\pi_{L[\mathcal{R}], \infty}(\gamma) =
\alpha$ and $\pi_{L[\mathcal{R}], \infty}(\gamma_0) = \alpha_0$. Thus,
$\alpha \in a$. 

This finishes the verification of the $\game \Gamma_{1, k+1} (\alpha_0)$
definition of $a$. Hence, $a$ is $u_{k+3}\text{-}\Pi^1_3(\alpha_0)$ by Lemma~\ref{lem:pointclass_transfer}. Hence, $a \in L_{u_{\omega}}[\mathcal{O}_{\Sigma^1_3}]$.
\end{proof}

\section{$u_n$ is in $\ran(\pi_{M_1,\infty})$}
\label{sec:u_n}
This section proves an interesting result that
\begin{displaymath}
  \text{for any $n<\omega$, } u_n \in \ran(\pi_{M_1, \infty}).
\end{displaymath}

It requires an ingredient from $Q$-theory. 
A major feature of $Q$-theory is the discrepancy between $\Delta^1_3$-degrees and $Q_3$-degrees: The universal $\Pi^1_3$ subset of $\omega$ is in $L_{\kappa_3}[T_2]$. In the spirit of its proof,  in \cite[Lemma 8.2]{Q_theory}, we establish a series of results along the same line. 

Define $\Delta_1^{L_{\kappa_3}[T_2]}(T_2)$ set $W_n$ where 
\begin{displaymath}
  \gamma \in W_n
\end{displaymath}
iff there is a $\Sigma_1$-formula $\varphi$ and an ordinal $\alpha < u_{n}$ such that
\begin{displaymath}
  L_{\gamma}(T_2) \models \varphi (\alpha, T_2) \wedge \forall \gamma' < \gamma(  L_{\gamma'}(T_2) \models \neg\varphi (\alpha, T_2)).
\end{displaymath}
Let
\begin{displaymath}
  \nu_n = \ot (W_n). 
\end{displaymath}
$W_n$ is therefore $\Delta_1^{L_{\kappa_3}[T_2]}(T_2, \nu_n)$. 

\begin{lemma}
  \label{lem:W_n_order_type}
Assume $\boldDelta{2}$-determinacy.  $\nu_n$ equals to the supremum of
the lengths of $\Delta^1_3(<\!u_{n})$ wellorderings on $u_n$. 
\end{lemma}
\begin{proof}
  Fix any $\gamma \in W_n$. Let $\alpha<u_n$ and let $\varphi$ be $\Sigma_1$ such that 
\begin{displaymath}
  L_{\gamma}(T_2) \models \varphi (\alpha, T_2) \wedge \neg \exists \gamma' (L_{\gamma'}(T_2) \models \varphi(\alpha, T_2)).
\end{displaymath}
$W_n \cap \gamma$ is then the length of a $\Delta^1_3(\alpha)$ prewellordering
\begin{displaymath}
\leq_A
\end{displaymath}
of a $\Delta^1_3(\alpha)$ subset
\begin{displaymath}
A \subseteq \mathrm{Fml}_{\Sigma_1} \times u_n
\end{displaymath}
where
\begin{displaymath}
  (\gcode{\psi}, \beta) \in A
\end{displaymath}
iff 
\begin{displaymath}
  L_{\kappa_3}(T_2) \models \exists \gamma' (L_{\gamma'}(T_2) \models (\psi (\beta, T_2) \wedge \neg \psi(\alpha, T_2) ))
\end{displaymath}
iff 
\begin{displaymath}
  L_{\kappa_3}(T_2) \models \forall \gamma' (L_{\gamma'}(T_2) \models (  \psi (\beta, T_2) \vee \neg \psi(\alpha, T_2) ),
\end{displaymath}
and for $(\gcode{\psi},\beta), (\gcode{\psi'},\beta') \in A$,
\begin{displaymath}
  (\gcode{\psi},\beta) \leq_A (\gcode{\psi'},\beta')
\end{displaymath}
iff the least $\gamma$ with $L_{\gamma}(T_2)\models \psi(\beta,T_2)$ is not greater than the least $\gamma$ with $L_{\gamma}(T_2) \models \psi'(\beta',T_2)$. From $\leq_A$ we can easily define a $\Delta^1_3(\alpha)$ wellordering on $u_n$ of the same order type. This shows one direction of the lemma.

On the other hand, we need to show that if $<^{*}$ is a $\Delta^1_3(<\!u_n)$-wellordering of $u_n$, then its length is smaller than $\nu_n$. We define a $\Sigma_1^{L_{\kappa_3}[T_2]}(T_2)$ partial function
\begin{displaymath}
  f
\end{displaymath}
by induction on $<^{*}$. 
Let $\varphi$ and $\psi$ be $\Sigma_1$ formulas such that $\alpha <^{*}
\beta$ iff $L_{\kappa_3}[T_2] \models \varphi(T_2, \alpha, \beta)$ iff
$L_{\kappa_3}[T_2] \models \neg \psi(T_2,\alpha, \beta)$. Let $\xi_0$ be
the smallest such that $L_{\kappa_0}(T_2) \models \forall
\alpha,\beta<u_n (\varphi(T_2, \alpha, \beta) \vee \psi(T_2, \alpha,
\beta))$.  
Suppose that $f (\beta) $ for $\beta <^{*} \alpha$ has been
defined. We let $f(\alpha)$ be the smallest $\xi> \xi_0$ such that
$L_{\xi}(T_2) \models ``f (\beta)$ is defined for any $\beta$
satisfying $\varphi(T_2, \beta, \alpha)$''. By admissibility, $f$ is a
total function from $u_n$ into $W_n$ and is order preserving with
respect to $<^{*}$ and $<$. This implies that the order type of $<^{*}$ is smaller than $\nu_n$. 
\end{proof}

\begin{lemma}
\label{lem:sigma1_to_pi1_in_para}
Assume $\boldDelta{2}$-determinacy. Fix $n<\omega$. If $A \subseteq u_n$ is $\Pi^1_3$, then $A$ is $\Delta^1_3(\nu_n)$, uniformly in the $\Pi^1_3$-definition of $A$.
\end{lemma}
\begin{proof}
Suppose that for $\alpha < u_n$, 
\begin{displaymath}
  \alpha \in A \eqiv L_{\kappa_3}(T_2) \models \varphi (\alpha, T_2)
\end{displaymath}
where $\varphi$ is $\Sigma_1$. 
Note that $W_n$ is $\Delta_1^{L_{\kappa_3}(T_2)}(\nu_n,T_2)$ and in particular, $W_n \in L_{\kappa_3}(T_2)$. 
Then,
\begin{displaymath}
  \alpha \in A \eqiv L_{\sup(W_n)}(T_2) \models \varphi (\alpha, T_2).
\end{displaymath}
This definition of $A$ is uniformly $\Delta_1^{L_{\kappa_3}(T_2)}(\nu_n, T_2)$. 
\end{proof}

The next lemma defines $\nu_n$ from $\se{u_1,\dots,u_n}$ in $L[x]$ for a cone of $x$, uniformly. The defining formula is called $\varphi_{v = \nu_n}(v, u_1, \dots, u_n)$. 
\begin{lemma}
  \label{lem:nu_definable}
Assume $\boldDelta{2}$-determinacy. There is a formula in the language of set theory
\begin{displaymath}
  \varphi_{ v= \nu_n} ( v , u_1, \dots, u_n)
\end{displaymath}
such that 
  for a cone of $x$,  
  \begin{displaymath}
    L[x] \models \forall v ~ ( \varphi_{ v= \nu_n} (v, u_1,\dots,u_n) \eqiv v = \nu_n).
  \end{displaymath}
\end{lemma}
\begin{proof}
The $\Delta^1_3(<\!u_n)$ subsets of $u_n^2$ have a universal coding, indexed by a $\Pi^1_3$ set. 
That is, there is a $\Pi^1_3$ set $A$ consisting of $(\gcode{\varphi},\gcode{\psi},\alpha)$ satisfying that
  \begin{enumerate}
  \item $\varphi,\psi$ are ternary $\Pi^1_3$-formulas, uniform in the sharp codes in all coordinates, defining $a \subseteq u_n^3$ and $b \subseteq u_n^3$ respectively,
  \item $\alpha < u_n$, and
  \item $c_{(\gcode{\varphi},\gcode{\psi},\alpha)} \DEF \set{(\beta,\beta')}{(\alpha,\beta,\beta') \in a} = u_n^2 \setminus \set{(\beta,\beta')}{(\alpha,\beta,\beta') \in b}$,
  \end{enumerate}
and such that for any $\Delta^1_3(<\!u_n)$ subset $d \subseteq u_n^2$, there is $(\gcode{\varphi},\gcode{\psi},\alpha) \in A$ such that $c_{(\gcode{\varphi},\gcode{\psi},\alpha)} = d$. 
Therefore, $\nu_n$ is the smallest $\nu$ with the $\Sigma^1_3$-property that
\begin{quote}
  for any $(\gcode{\varphi},\gcode{\psi},\alpha) \in A$, if $c_{(\gcode{\varphi},\gcode{\psi},\alpha)}$ is a wellordering on $u_n$, then its length is smaller than $\nu_n$.
\end{quote}
Extract a $\Sigma^1_3$-formula 
\begin{displaymath}
\psi_n(w)
\end{displaymath}
from this $\Sigma^1_3$-property. That is, $\psi_n(w)$ holds iff
\begin{displaymath}
w \in \WO_{n+1}  \wedge   \sharpcode{w} \geq \nu_n.
\end{displaymath}
Pick $w_0 \in \WO_{n+1}$ with $\sharpcode{w_0} = \nu_n$ and pick $x_0$ witnessing the existence quantifier of the $\Sigma^1_3$-definition of $\psi_n(w_0)$. 
Then for any $x \geq w_0 \oplus x_0$, $L[x]$ satisfies
\begin{quote}
  ``$\nu_n$ is the smallest ordinal such that for some $w \in \WO_{n+1}$, $\sharpcode{w} = \nu_n$ using $(u_1,\dots, u_n)$ to evaluate $\sharpcode{w}$, and $\psi_n(w)$ holds''.
\end{quote}
This is the definition of $\varphi_{v = \nu_n}$. 
\end{proof}

\begin{lemma}
  \label{lem:mu-n-definable}
Suppose that $N$ is a countable iterate of $M_1$ such that the
iteration map on the main branch exists. Then for any $n$, $\nu_n$ is
uniformly definable over $N$ from $\se{u_1, \dots, u_n}$. 
\end{lemma}
\begin{proof}
Let $P$ be a countable iterate of $N$ via the iteration map $\pi_{NP}$ such that the base of the cone
in Lemma~\ref{lem:nu_definable} is in a $\coll(\omega,
\delta^P)$-extension of $P$. Then 
  \begin{displaymath}
    P^{\coll(\omega, \delta^P)}\models \forall v ~ ( \varphi_{ v= \nu_n} (v, u_1,\dots,u_n) \eqiv v = \nu_n).
  \end{displaymath}
By elementarity, $\nu_n \in \ran(\pi_{NP})$.
Thus, if $g$ is $\coll(\omega, \delta^N)$-generic over $N$, there is $w
\in \WO_{n+1} \cap N[g]$ such that $\sharpcode{w} = \pi_{NP}^{-1}(\nu_n)$ and
$\psi_n(w)$ holds in $N[g]$. By $\Sigma^1_3$-upward absoluteness,
$\psi_n(w)$ holds in $V$. Thus,  $\sharpcode{w} \geq \nu_n$. Of
course, $\pi_{NP}$ is non-decreasing. Thus, $\sharpcode{w} = \nu_n =
\pi_{NP}(\nu_n)$. The uniform definition of $\nu_n$ is 
  \begin{displaymath}
    N^{\coll(\omega, \delta^N)}\models \forall v ~ ( \varphi_{ v= \nu_n} (v, u_1,\dots,u_n) \eqiv v = \nu_n).
  \end{displaymath}
\end{proof}

\begin{theorem}
  Assume $\boldDelta{2}$-determinacy. Then for any $n<\omega$,
  \begin{displaymath}
    u_n \in \ran (\pi_{M_1, \infty}). 
  \end{displaymath}
\end{theorem}
\begin{proof}
  Recall the function $f_n : \delta_n \to P_{u_1, \dots, u_n}$ in
  Section~\ref{sec:full-direct-limit}. We argued from
  Lemmas~\ref{lem:P_n_complexity}-\ref{lem:pointclass_transfer} that
  $P_{u_1, \dots, u_n}$ is $u_{n+3}\text{-}\Pi^1_3$. By
  Lemma~\ref{lem:sigma1_to_pi1_in_para}, $P_{u_1, \dots, u_n}$ is
  $\Delta^1_3(\nu_{n+3})$. Hence, $f_n$ is
  $\Delta^1_3(\nu_{n+3})$. A similar proof to
  Lemma~\ref{lem:LT2_in_M1infty} yields that for any $\beta <
  \delta_n$,  
\begin{multline*}
  M_{1,\infty} ^{\coll(\omega, \delta_{1,\infty})}\models 
 f^{G(\nu_{n+3})}_{n,G(u_1),\dots,G(u_{n+3})} (G(\beta)) =\rho^{M_{1,\infty} | G(u_{n+1})}(\beta, (G(u_1),\dots, G(u_n))) ,
\end{multline*}
where
\begin{align*}
  f^{\nu}_{n, \kappa_1, \dots, \kappa_{n+3}} =   \{ (\alpha, \beta): ~& \exists w , z_1, z_2 ~( \varphi_{B_n^{*}}(w, z_1, z_2)  \wedge \\
    &\sharpcode{w} = \nu \wedge \sharpcode{z_1} = \alpha \wedge \sharpcode{z_2} = \beta \\
 &     \text{using $\kappa_1, \dots, \kappa_{n+3}$ to evaluate $\sharpcode{w}, \sharpcode{z_1}, \sharpcode{z_2}$} ). 
\end{align*}
where $\varphi_{B_n^{*}}(w,z_1, z_2)$ is a $\Sigma^1_3$-defining formula of the $\Sigma^1_3$ set
\begin{displaymath}
  B_n^{*} \subseteq u_{n+4} \times (u_{n+1}
\times u_{n+1})
\end{displaymath}
such that $\varphi(w,z_1,z_2)$ iff $w,z_1,z_2 \in \WO_{n+4} \wedge (\sharpcode{w},\sharpcode{z_1},\sharpcode{z_2}) \in B_n^{*}$, and  
\begin{displaymath}
  f_n(\alpha) = \beta \eqiv (\nu_{n+3}, (\alpha, \beta)) \in B_n^{*}. 
\end{displaymath}
In particular,
as  $u_n < \delta_n$ by Hjorth in   \cite{hjorth_boundedness_lemma} ,
$u_n$ is definable over $M_{1, \infty}$ from $\{G(u_1), \dots,
  G(u_{n+3}), G(\nu_{n+3})\}$. 
By Lemma~\ref{lem:mu-n-definable}, $G(\nu_{n+3})$ is definable over
$M_{1,\infty}$ from $\{G(u_1), \dots, G(u_{n+3})\}$. Thus, $u_n$ is definable over $M_{1, \infty}$ from $\{G(u_1), \dots,
  G(u_{n+3})\}$. Finally, because $G(u_i) = \pi_{M_1, \infty}(u_i)$ for any
  $i$, $u_n \in \ran(\pi_{M_1, \infty})$. 
\end{proof}

\section{Open questions}
\label{sec:open-questions}

An interesting question is the indiscernibility of $(u_n:n\geq 3)$ in $M_{1,\infty}$.
\begin{conjecture}
  Suppose $A \subseteq u_{\omega}$ is in $M_{1,\infty}$. Then there is $m<\omega$ such that either 
  \begin{displaymath}
    \set{u_n}{m<n<\omega} \subseteq A
  \end{displaymath}
or 
  \begin{displaymath}
    \set{u_n}{m<n<\omega} \cap  A = \emptyset.
  \end{displaymath}
\end{conjecture}

The $\kappa_3^x$ ordinal in \cite{Q_theory} might have an explanation via inner model theory. A candidate is the sequence $((u_n^{+})^{M_{1,\infty}(x)} : n<\omega)$  modulo the Fr\'{e}chet filter.
\begin{conjecture}
  $\kappa_3^x \leq \kappa_3^y$ iff there is $m<\omega$ such that for any $m<n<\omega$,
  \begin{displaymath}
    (u_n^{+})^{M_{1,\infty}(x)}  \leq (u_n^{+})^{M_{1,\infty}(y)} .
  \end{displaymath}
\end{conjecture}

The uniqueness of $L[T_2]$, solved by Hjorth in \cite{hjorth_LT2}, has a local version which is more to the point, as $M_{1,\infty}$ is a mouse. 
\begin{question}
  Suppose that $(\psi_n : n<\omega)$ is  a $\Delta^1_3$-scale on a good universal $\Pi^1_2$ set such that each $\psi_n$ is $\game(<\!\omega^2\text{-}\Pi^1_1)$. Define
  \begin{displaymath}
    \mathcal{O}_{\Sigma^1_3, \vec{\psi}} = \set{ (n,\alpha, \gcode{\varphi})}{\varphi \text{ is } \Sigma^1_3, \exists x~(\psi_n(x) = \alpha \wedge \varphi(x))}.
  \end{displaymath}
Must
\begin{displaymath}
  L_{u_{\omega}}[\mathcal{O}_{\Sigma^1_3,\vec{\psi}}] = M_{1,\infty} | u_{\omega}?
\end{displaymath}
\end{question}

\section*{Acknowledgements}
\label{sec:acknoledgements}
This paper was partially written during Oberwolfach Set Theory Workshop, February 2017. The author thanks Steve Jackson for many helpful conversations. 

The author also thanks the referee for many helpful suggestions, including a note which results in a simplification of the original argument. The orginal arguments works only under $\boldDelta{3}$-determinacy. Now it becomes much clearer and works under the weaker hypothesis of $\boldDelta{2}$-determinacy.
\bibliography{m1}
\bibliographystyle{hplain}
\end{document}